\documentclass[10pt]{article}

\usepackage{mathptmx}
\usepackage{wrapfig}
\usepackage{eucal}
\usepackage{amsmath}
\usepackage{amscd}
\usepackage{amssymb}
\usepackage{amsthm}
\usepackage{xspace}
\usepackage[all,tips]{xy}
\usepackage[dvips]{graphicx}
\usepackage{graphics}
\usepackage{epsfig}
\usepackage{verbatim}
\usepackage{syntonly}
\usepackage{hyperref}
\usepackage{amssymb, color, url}

\newtheorem{theorem}{Theorem}[section]
\newtheorem{proposition}[theorem]{Proposition}
\newtheorem{definition}[theorem]{Definition}

\newtheorem{corollary}[theorem]{Corollary}
\newtheorem{lemma}[theorem]{Lemma}

\input xy
\xyoption{all}

\DeclareMathOperator{\comIndex}{c}

\DeclareMathOperator{\Z}{\mathbb{Z}}

\newcommand{\CL}{\mathcal L}

\newcommand{\comGrowth}{{\bf c}}

\newcommand{\CPp}{\CL_{(p)}}
\newcommand{\CT}{\mathcal T} 

\newcommand{\nilA}{\mathcal{U}}
\newcommand{\setU}{\mathfrak{U}}

\def\R{\mathbb{R}}

\def\Q{\mathbb{Q}}

\def\Z{\mathbb{Z}}
\def\N{\mathbb{N}}

\usepackage{fancyhdr}

\fancypagestyle{plain}

\begin{document}

%-----------------------------------------------------------
%-----------------------------------------------------------

\title{\textbf{Arithmetic lattices in unipotent algebraic groups}}
\author{Khalid Bou-Rabee\thanks{K.B. supported in part by NSF grant
    DMS-1405609} \and Daniel Studenmund\thanks{D.S. supported in part
    by NSF grant DMS-1246989}}
\maketitle

%-----------------------------------------------------------
%-----------------------------------------------------------

\begin{abstract}
Fixing an arithmetic lattice $\Gamma$ in an algebraic group $G$, the commensurability growth function assigns to each $n$ the cardinality of the set of subgroups $\Delta$ with $[\Gamma : \Gamma \cap \Delta] [\Delta: \Gamma \cap \Delta] = n$. 
This growth function gives a new setting where methods of F.~Grunewald, D.~Segal, and G.~C.~Smith's ``Subgroups of finite index in nilpotent groups" apply to study arithmetic lattices in an algebraic group.
In particular, we show that for any unipotent algebraic $\Z$-group with arithmetic lattice $\Gamma$, the Dirichlet function associated to the commensurability growth function satisfies an Euler decomposition. Moreover, the local parts are rational functions in $p^{-s}$, where the degrees of the numerator and denominator are independent of $p$. This gives regularity results for the set of arithmetic lattices in $G$.
\end{abstract}
\vskip.1in
{\small{\bf keywords:}
\emph{commensurability, lattices, unipotent algebraic groups, subgroup growth}}

\section{Introduction}

Let $G$ be an algebraic group defined over $\Z$ (an \emph{algebraic $\Z$-group}).
Two subgroups $\Delta_1$ and $\Delta_2$ of $G(\R)$ are
\emph{commensurable} if their \emph{commensurability index} 
\[
  \comIndex(\Delta_1, \Delta_2) := [\Delta_1 : \Delta_1 \cap
  \Delta_2][\Delta_2 : \Delta_1 \cap \Delta_2]
\] 
is finite.
An \emph{arithmetic lattice} of $G$ is a subgroup of $G(\R)$ that is commensurable with $G(\mathbb{Z})$.
The first purpose of this article is to show that when $G$ is
unipotent, the set of arithmetic lattices in $G$ has a great deal of
regularity. The second purpose is to bring attention to a new notion of quantifying commensurability.

The main tool of this article is the {\em commensurability growth function} $\N \to \N \cup \{ \infty \}$ assigning to each $n \in \N$ the cardinality
 $$\comGrowth_n(G(\Z), G(\R)) := | \{ \Delta \leq G(\R) : \comIndex\left(G(\Z), \Delta \right) = n \} |.$$
  We study $\comGrowth_n(G(\Z), G(\R))$ for
 groups $G$ in the class $\nilA$ of unipotent algebraic $\Z$-groups,
 starting with the fact that $\comGrowth_n(G(\Z), G(\R))$ is finite for all $n$ in
 Lemma \ref{lem:finite}.  Note that the map $G \mapsto G(\Z)$ gives the Mal'cev correspondence \cite{MR0349383} between torsion-free finitely generated nilpotent groups and unipotent $\Z$-groups (see also \cite{MR797867}).

Our method of attack is that of F.~Grunewald, D.~Segal, and G.~C.~Smith \cite{MR943928}, who studied the {\em subgroup growth function} of torsion-free finitely generated nilpotent groups $\Gamma$, defined by 
\[
a_n(\Gamma) := \{ \Delta \leq \Gamma : [\Gamma : \Delta]=n \},
\]
by decomposing the associated zeta function into a product of local
functions (see also \cite{MR1217350}). In view of the relationship $\comGrowth_n(\Gamma, \Gamma) = a_n(\Gamma)$, the function $\comGrowth_n$ extends $a_n$ to pairs of
groups. While we have the simple
relationship $\comGrowth_n(G(\Z), G(\R)) \geq a_n (G(\Z)),$ it is
clear from the results in this paper that this inequality is strict
when $G$ is unipotent. 
The difference is evident even in the one-dimensional case.
%, and is amplified
%in higher dimensions.
% We expect further work to nail down this difference. <<--- Is this true?

\subsection{The setting}

Our results about $\comGrowth_n$ are stated in the language of zeta functions.
Let $G$ be an algebraic group and let $\mathcal{F}$ a family of subgroups of $G(\R)$.
We associate to such a family the Dirichlet series
\begin{equation} \label{eqn:zetadef}
\zeta_{\mathcal{F}}(s) = \sum_{\Delta \in \mathcal{F}} \comIndex(G(\Z), \Delta)^{-s} = \sum_{n=1}^\infty \comGrowth_n(\mathcal{F}) n^{-s}
\end{equation}
where 
\[
  \comGrowth_n(\mathcal{F}) := | \left\{ \Delta \in \mathcal{F} :
    \comIndex(G(\Z), \Delta) = n \right\} |.
\]

Following \cite{MR943928}, there are two questions one can ask about
the sequence $\comGrowth_n(\mathcal{F})$: (a) how fast does it grow,
and (b) how regularly does it behave? Pursuing (a) is the study of
{\em commensurability growth}, which we do not directly address here. We will address (b) by decomposing $\zeta_{\mathcal{F}}$ into local parts and proving a local regularity theorem. The families we consider are the following:
\begin{eqnarray*}
\CL(G) &:=& \{ \text{all arithmetic lattices inside } G \}, \\
\CL_n(G) &:=& \{ \Delta \in \CL(G) : c(\Delta, G(\Z)) = n \},\\
\CPp(G) &:=& \{ \Delta \in \CL(G) : c(\Delta, G(\Z)) \text{ is a power of $p$} \}.
\end{eqnarray*}
\noindent
Define the global and local {\em commensurability zeta functions}:
\[
\zeta_G := \zeta_{\mathcal{L}(G)} \qquad \text{ and } \qquad \zeta_{G,p} :=
\zeta_{\CPp(G)}.
\]
%\begin{eqnarray*}
%\zeta_G := \zeta_{\mathcal{L}(G)}, & \alpha_G := \alpha_{\mathcal{F}(G)}. 
%\\
%%\zeta_\Gamma^\lhd = \zeta_{\mathcal{N}(G)}, & \alpha_G^\lhd = \alpha_{\mathcal{N}(G)}, \\
%\zeta_{G,p} := \zeta_{\CPp(G)}, & \alpha_{G,p} := \alpha_{\CPp(G)}.
%%\zeta_{\hat \Gamma} = \zeta_{\mathcal{H}(G)}, & \hat \alpha_G = \alpha_{\mathcal{H}(G)}, \\
%\end{eqnarray*}

\subsection{The main results}

We start with an elementary example, proved in \S \ref{sec:additivecase}.
\begin{proposition} \label{thm:MAIN}
Let $G = G_a$, the additive algebraic group, so that $G(\Z) = \Z$ and $G(\R) = \R$. Then formally,
$$
\zeta_G(s) = \frac{\zeta^2(s)}{\zeta(2s)}.
$$
\end{proposition}

The zeta functions introduced in \cite{MR943928} decompose into Euler products.
Our first main result, proved in \S \ref{subsection:eulerdecomp}, is an
analogous decomposition of the commensurability zeta function $\zeta_G$ for $G\in \nilA$ .

\begin{proposition} \label{prop:eulerd}
If $G \in \nilA$, then 
\begin{eqnarray*}
\zeta_G(s) &=& \prod_p \zeta_{G,p}(s)
\end{eqnarray*}
as formal products over all primes of Dirichlet series.%, convergent for $\Re(s) > \alpha_G$.
\end{proposition}

The main content of this article, our next result, is proved in \S \ref{subsection:maintheoremproof}:

\begin{theorem} \label{thm:main}
Let $G \in \nilA$, then the function $\zeta_{G,p}(s)$ is rational in $p^{-s}$, where the degrees of the numerators and denominators are bounded independently of the prime $p$.
\end{theorem}

Proposition \ref{prop:eulerd} and Theorem \ref{thm:main} provide regularity results for the set of arithmetic lattices, $\CL(G)$.

\begin{corollary}
Let $G \in \nilA$, then
\begin{enumerate}
\item $\comGrowth_n(\CL(G)) = \prod \comGrowth_{p_i^{e_i}}(\CL(G))$ where $n = \prod p_i^{e_i}$ is the prime factorization of $n$;
\item there exists positive integers $l$ and $k$ such that, for each
  prime $p$, the sequence $\left( \comGrowth_{p^i}(\CL(G)) \right)_{i > l}$ satisfies a linear recurrence relation over $\Z$ of length at most $k$.
\end{enumerate}
\end{corollary}

Our proof of Theorem \ref{thm:main} uses \cite[Theorem 22]{MR1076249}, which is a stronger rationality theorem than that used in \cite{MR943928}. Our setup requires such a theorem since we express the local commensurability zeta function as a $p$-adic integral over an unbounded set (see Proposition \ref{prop:zetaintegral}). To find this explicit formula, we constructed a parametrization of lattices that gives a closed form for the commensurability index function (see Lemma \ref{lem:comindex}). We also discovered an explicit correspondence 
$$
\{\text{closed }K \leq G(\Q_p) : \comIndex(G(\Z_p),K) = p^k \} \leftrightarrow
\{ \Delta \leq G(\Q) : \comIndex(G(\Z), \Delta) = p^k \},
$$
that may be of independent interest (see Lemma
\ref{lem:correspondence}). 

We were moved to pursue this subject after reading the work of
N.~Avni, S.~Lim, and E.~Nevo \cite{MR2897732} on the related concept
of {\em commensurator growth}.
Commensurator growth in \cite{MR2897732}, subgroup growth in
\cite{MR943928}, and commensurability growth in this paper are all
studied through their associated zeta functions.
Associating zeta functions to growth functions in groups is an active
area of research, centering around subgroup growth and representation
growth. For background reading on these subjects we recommend the
references \cite{MR1978431, MR2807857, MR2807855}.

\subsubsection*{\emph{Notation}}

$\comIndex(A,B) = [A : A \cap B][B: A \cap B]$.\\

\noindent
$\nilA = $ the class of unipotent algebraic $\Z$-groups.\\
$\CT = $ the class of torsion-free finitely generated nilpotent groups.\\

\noindent
$|S| =$ the cardinality of the set $S$.\\
$[x,y] = x^{-1} y^{-1} x y = x^{-1} x^y$. \\
%$[S,T] = \left< [s,t] : s \in S, t \in T \right>$. \\
$G^n = \left< g^n : g \in G \right>$ for $n \in \Z$. \\
$Z(G) = $ the center of $G$.\\

\noindent
$T_n(R) = $ set of lower-triangular $n\times n$ matrices over a
commutative ring $R$. \\

\noindent
$\zeta(s)$ is the Riemann zeta function.

\paragraph{Acknowledgements} We are grateful to Benson Farb, Tasho
Kaletha, Michael Larsen, and Andrew Putman for their conversations and
support. In particular, Benson Farb provided helpful comments on an early draft.

\section{\emph{The 1-dimensional case}} \label{sec:additivecase}

We begin with the integers. In this case we can directly relate the commensurability zeta function with the classical zeta function.

\begin{proposition}
Let $G = G_a$, the additive algebraic group, so that $G(\Z) = \Z$ and $G(\R) = \R$. Then formally,
$$
\zeta_G(s) = \frac{\zeta^2(s)}{\zeta(2s)}.
$$
\end{proposition}

\begin{proof}
The subgroups of $\R$ commensurable with $\Z$ are all of the form  $r \Z$ where $r \in \Q^*$.
Writing $r = n/d$ in reduced form, we have
$\comIndex(\Z, r \Z) = n d.$
Hence,
$$
\comGrowth_n(\CL(G)) = |\{ a/b : a, b \in \N, (a,b) = 1, ab = n \} |.
$$
From this, we get for distinct primes $p_1, p_2, p_3, \ldots, p_n$,
$$
\comGrowth_{p_1^{j_1} p_2^{j_2} \cdots p_n^{j_n}}(\CL(G)) = \comGrowth_{p_1^{j_1}}(\CL(G)) \comGrowth_{p_2^{j_2}}(\CL(G)) \cdots \comGrowth_{p_n^{j_n}}(\CL(G)).
$$
And for any prime $p$,
$$
\comGrowth_{p^k} (\CL(G)) = 2.
$$
Hence,
$$
\comGrowth_n(\CL(G)) = 2^{\omega_n},
$$
where $\omega_n$ is the number of distinct primes dividing $n$.
Following \cite[p. 255]{MR2445243}, we compute:
\begin{eqnarray*}
\zeta_G(s) &=& \sum_{n=1}^\infty 2^{\omega_n} n^{-s} \\
&=& \prod_{p \text{ prime}} \left( 1 + \sum_{k=1}^\infty \frac{2}{p^{ks}} \right) \\
&=& \prod_{p \text{ prime}} \left( \frac{2p^{-s}}{1-p^{-s}} +1 \right) \\
&=& \prod_{p \text{ prime}} \frac{1-p^{-2s}}{(1-p^{-s})^2}\\
&=& \frac{\zeta^2(s)}{\zeta(2s)}.
\end{eqnarray*}
\end{proof}

\section{Unipotent algebraic groups} \label{sec:unipotentgroups}

To any $G \in \nilA$ we fix an embedding of $G(\R)$ into a group $T_n(\R)$ such that $G(\Q) = T_n(\Q) \cap G(\Q)$ given by Kolchin's Theorem \cite{MR0024884}.
With this in hand, we first show that for any unipotent algebraic $\Z$-group $G$, the commensurability growth function takes values in $\N$:

\begin{lemma} \label{lem:finite}
Let $G$ be in $\nilA$.
Then $\CL_n(G)$ is finite for any $n \in \N$.
\end{lemma}

\begin{proof}
Given $n \in \N$, for any $\Delta \in \CL_n(G)$ and for every $h \in \Delta$, we have $h^n \in G(\Z)$.
Then by \cite[Exercise 7, page 114]{MR713786}, there exists a finitely generated subgroup $\Gamma \leq G(\Q)$ that contains every $n$th root of an element of $G$.
Then every element in $\CL_n(G)$ is a subgroup of $\Gamma$ of index at most $n[\Gamma: G(\Z)]$.
Since $\Gamma$ is finitely generated, it has finitely many subgroups of any given index, so $\CL_n(G)$ is finite.
\end{proof}

As an immediate consequence of the proof of Lemma \ref{lem:finite}, we get the well-known fact that arithmetic lattices in $G$ are contained in $G(\Q)$:

\begin{lemma} \label{lem:lattices}
Let $G$ be in $\nilA$.
Then each element of 
$\mathcal{L}(G)$ is a subgroup of $G(\Q)$. \qed
\end{lemma}

Slight modifications of a technical idea in the proof of Lemma
\ref{lem:finite} will be used repeatedly throughout the rest of the paper, so we encapsulate them here.
Throughout, for any $\Delta \leq G(\R)$ and $n \in \N$, we write $\Delta^{1/n}$ to be
the group generated by elements $g \in G(\R)$ such that $g^n \in \Delta$.

\begin{lemma} \label{lem:roots}
Let $G$ be in $\nilA$ and $\Gamma = G(\Z)$ have nilpotence class $c$ and Hirsch length $d$.
\begin{enumerate}
\item \label{item:rootsa}
For any $k \in \N$, $[\Gamma^{1/k} : \Gamma]$ is finite. 

\item \label{item:rootsb}
We have $(\Gamma^{p^{-k}})^{p^{k c(c+1)/2}} \leq \Gamma$. 

\item \label{item:rootsc}
For any prime $p$ and any $k \in \N$, we have $[\Gamma^{p^{-k}}:\Gamma] \leq p^{dk c(c+1)/2}.$

\end{enumerate}
\end{lemma}

\begin{proof}
Item \ref{item:rootsa} follows from \cite[Exercise 7, page 114]{MR713786}.
Item \ref{item:rootsb} follows from \cite[Proposition 3, pp. 113]{MR713786}.
For the final statement, by \cite[Exercise 7, page 114]{MR713786} the group $\Gamma^{p^{-k}}$ is a finitely generated nilpotent group.
Moreover, by being in $G(\Q)$, it has Hirsch length at most $d$, and hence
$$
[\Gamma^{p^{-k}} : (\Gamma^{p^{-k}})^{p^{k c(c+1)/2}} ] \leq p^{d k c(c+1)/2}.
$$
Then by Item \ref{item:rootsb} it follows that we have bounded $[\Gamma^{p^{-k}} : \Gamma]$ as desired.
\end{proof}

\subsection{The Euler decomposition} \label{subsection:eulerdecomp}

\begin{proposition} \label{prop:localtoglobalzeta}
If $G$ is in $\nilA$, then 
\begin{eqnarray*}
\zeta_G(s) &=& \prod_{p \text{ prime}} \zeta_{G,p}(s).
\end{eqnarray*}
\end{proposition}

\begin{proof}
Any group in $\CL_n(G)$ is a subgroup of $G(\Z)^{1/n}$.
Hence, by applying Item \ref{item:rootsa} of Lemma \ref{lem:roots}, we
see that $\Omega = \left< \CL_n(G) \right>$ is finitely generated and
contains each element of $\mathcal{L}_n(G)$ as a subgroup of finite index. Hence, $\Lambda_n' = \cap_{A \in \CL_n(G)} A$ is a subgroup of finite index in $\Omega$. Let $\Lambda_n$ be the normal core of $\Lambda_n'$ in $\Omega$.

Consider the group $\Gamma := \Omega/\Lambda_n$. It is a finite nilpotent group, and hence decomposes into a product of its Sylow $p$-subgroups:
$$
\Omega/\Lambda_n = \prod_p S_p.
$$
For any image $A$ in $\Gamma$ of an element $\Delta \in \mathcal{L}_n(G)$, we get the following decomposition:
$$
A = \prod_p A_p,
$$
where $A_p$ are the Sylow $p$-subgroups of $A$ and $A_p \leq S_p$ for every $p$.
A similar statement is true for the image $Q$ of $G(\Z)$ in $\Gamma$, we have 
$$
Q = \prod_{p \text{ prime}} Q_p
$$
where $Q_p \leq S_p$.
We compute,
$$
Q \cap A = \prod_{p  \text{ prime}} (A_p \cap Q_p).
$$
Hence, 
$$
\comIndex(G(\Z), \Delta) = \comIndex(Q, A) = \prod_{p  \text{ prime}} \comIndex(A_p, Q_p).
$$
Further, $\comIndex(A_p, Q_p)$ is the greatest power of $p$ that divides $n$.
Since any such element of $\CL_n(G)$ arises from such a decomposed $A$, we have
$$
\comGrowth_n(\mathcal{L}) = \prod_{p^k || n } \comGrowth_{p^k}(\mathcal{L}),
$$
and hence the Euler decomposition for the commensurability zeta function above holds.
\end{proof}

\subsection{$p$-adic formulation and the proof of Theorem \ref{thm:main}} \label{subsection:maintheoremproof}

Let $G \in \nilA$.
Fix a Mal'cev basis $(x_1, \ldots, x_n)$ for $G(\Z)$, so that 
\[
  1 < \left< x_1 \right> < \left< x_1, x_2 \right> < \dotsb < \left<
      x_1,\dotsc, x_n \right> = G(\Z)
\]
is a central series with infinite cyclic successive quotients. 
Elements in $G(\Q_p)$ may be identified with the set of all ``$p$-adic words'' of the form
$$
{\bf x} ({\bf a}) = x_1^{a_1} \cdots x_n^{a_n},
$$
where ${\bf a} \in \Q_p^n$. Note that ${\bf x}$ parametrizes $G(\Z)$
when restricted to ${\bf a} \in \Z^n$.  Let $\lambda$, $\kappa$, and $\mu^{(i)}$
for $i\geq 2$ be polynomials defined over $\Q$ as in
\cite[pp. 197]{MR943928} for the $\CT$ group $G(\Z)$, so that for
${\bf a}\in \Z^n$ we have
\begin{eqnarray*}
{\bf x}({\bf a})^k= {\bf x}(a_1, a_2, a_3, \ldots, a_n)^k &=& {\bf x}(\lambda({\bf a},k)), \\
{\bf x}({\bf a}_1) {\bf x}({\bf a}_2) \cdots x({\bf a}_i) &=& {\bf x}(\mu^{(i)}({\bf a}_1, {\bf a}_2, \ldots {\bf a}_i)),\\
\left[{\bf x}({\bf a}_1), {\bf x}({\bf a}_2)\right] &=& {\bf x}(\kappa({\bf a}_1, {\bf a}_2)).
\end{eqnarray*}

We denote the closure of a subset $S$ of $G(\Q_p)$ to be $\overline{S}$ or $S^-$. 
For any $\Gamma \in \CL(G)$ we have that $\overline{\Gamma}$ is the
pro-$p$ completion of $\Gamma$ \cite[Theorem 4.3.5]{MR1691054}.

\begin{lemma} \label{lem:correspondence}
  For each $k$, the closure map gives a one-to-one correspondence
  between elements of $\CL_{p^k}(G)$ and closed subgroups $H$ of
  $G(\Q_p)$ with $c(H, G(\Z_p)) = p^k$.
\end{lemma}
\begin{proof}
  Let $\Gamma = G(\Z)^{p^{-k}}$. Note that $\Gamma$ contains every
  subgroup in $\CL_{p^k}(G)$. Moreover, by Item \ref{item:rootsb} of
  Lemma \ref{lem:roots}, we have that $\Gamma^{p^{N}} \leq G(\Z)$ for
  $N = kc(c+1)/2$, where $c$ is the nilpotence class of $G$.  It
  follows that any $\Delta \in \CL_{p^k}(G)$ contains
  $\Gamma^{p^{N+k}}$. Note that $\Gamma^{p^{N+k}}$ is a normal
  subgroup of $\Gamma$ and the quotient group
  $\Gamma / \Gamma^{p^{N+k}}$ is a finite $p$-group. It follows that
  for each $\Delta \in \CL_{p^k}(G)$ there is some $\ell \in \N$ such
  that $[\Gamma : \Delta] = p^\ell$.

  Let $K$ be the closure of $\Gamma$ in $G(\Q_p)$. Then $K$ is a
  finitely generated pro-$p$ group. Any closed subgroup
  $H \leq G(\Q_p)$ satisfying $c(H, G(\Z_p)) = p^k$ is a finite-index
  subgroup of $G(\Z_p)^{p^{-k}} = K$. The closure map gives an
  index-preserving bijection between subgroups of $\Gamma$ with index
  a power of $p$ and finite-index subgroups of $K$. This completes the
  proof.
  % Let $N' \leq K$ be a
  % normal subgroup containing every such subgroup $H$, and let $N = N'
  % \cap (\Gamma^{p^N})^-$. 
\end{proof}

Define $G_i = \left< x_1^{r_1}, \ldots, x_i^{r_i} : r_1, \ldots, r_i \in \Q_p \right>$. The following is an
extension of the notion of good basis from \cite{MR943928} to the
$G(\Q_p)$ setting.
\begin{definition} \label{def:goodbasis}
Let $H$ be a closed subgroup of $G(\Q_p)$.
An $n$-tuple $(h_1, \ldots, h_n)$ of elements in $H$ is called a {\bf good basis} if for each $i = 1, \ldots, n$, 
$$
\left< h_1, \ldots, h_i \right>^- = H \cap G_i.
$$
\end{definition}

\begin{lemma} \label{lem:goodbasis}
Let $h_i \in G_i \setminus G_{i-1}$ for $i = 1,\ldots, n$ and let $H = \left< h_1, \ldots, h_n \right>^-$.
Then:
\begin{enumerate}
\item \label{item:goodbasis1}$\comIndex(H, G(\Z_p)) < \infty$.
\item\label{item:goodbasis2}  $(h_1, \ldots, h_n)$ is a good basis for $H$ if and only if
\begin{equation} \label{lem:eqn1}
[h_i, h_j] \in \left< h_1, \ldots, h_{i-1} \right>^- \text{ for } 1 \leq i < j \leq n.
\end{equation} 
\item\label{item:goodbasis3} Suppose Equation \ref{lem:eqn1} holds, and let $h_1', \ldots, h_n' \in H$.
Then $(h_1', \ldots, h_n')$ is a good basis for $H$ if and only if there exists $r_i \in \Z_p^*$ and $w_i \in \left< h_1, \ldots, h_{i-1}\right>^-$ such that
$$h_i' = w_i h_i^{r_i} \text{ for } 1 \leq i \leq n.$$
\end{enumerate}
\end{lemma}

\begin{proof}
Since $h_i \in G_i \setminus G_{i-1}$ for $i = 1, \dots, n$, we have
$$
h_i = x_1^{r_{i1}} x_2^{r_{i2}} \cdots x_n^{r_{in}},
$$
where $r_{ij} \in \Q_p$.
For any such nonzero $r_{ij}$ we have $r_{ij} = u p^{-k}$ where $k \in \Z$ and $u \in \Z_p^*$.
If all such $k$ are nonpositive, then we appeal to \cite[Lemma 2.1]{MR943928}.
Otherwise, let $k$ be the maximal integer that appears in this way.
Then 
$$H \leq \left< x_1^{p^{-k}}, x_2^{p^{-k}}, \ldots, x_n^{p^{-k}} \right>^- \leq G(\Z_p)^{p^{-k}}.$$
By continuity of taking powers, $G(\Z_p)^{p^{-k}}$ is the closure of $G(\Z)^{p^{-k}}$ in $G(\Q_p)$.
Thus, by applying Item \ref{item:rootsb} of Lemma \ref{lem:roots}, we have that there exists $m \in \N$ such that $H^m \leq G(\Z_p)$.
It follows that $G(\Z_p) \cap H$ must have the same dimension as $G$, and so Item \ref{item:goodbasis1} follows.

Items \ref{item:goodbasis2} and \ref{item:goodbasis3} follow from small modifications of the proof of \cite[Lemma 2.1]{MR943928}.
\end{proof}

\begin{definition} \label{def:maindef}
For $\Delta \in \CL(G)$, define $\setU(\Delta)$ to be the collection of pairs $(A, B)$ where $A \in T_n(\Q_p)$, $B \in T_n(\Z_p)$, such that, if $a_i$ denotes the $i$th row of $A$ and $b_i$ denotes the $i$th row of $B$, 
\begin{enumerate}
\item \label{item:def1} $({\bf x}(a_1), \ldots, {\bf x}(a_n))$ is a
  good basis for $\overline{\Delta}$, and
\item \label{item:def3} $({\bf x}(b_1), \ldots,{\bf  x}(b_n))$ is a good basis for $\overline{\Delta} \cap G(\Z_p)$.
\end{enumerate}
\end{definition}

Combining suitable $\setU(\Delta)$ into one set, we define $$\setU_p := \bigcup_{k=1}^\infty \left( \bigcup_{\Delta \in \CL_{p^k}(G)} \setU(\Delta) \right).$$ 
The next proposition shows that $\setU_p$ is $L_p$-defined in the sense of \cite{MR1076249} (that is, it can be defined using first-order logic, $p$-norms, and field operations). Note that $|x|_p \geq | y |_p$ can be stated in $L_P$ (see the example in \cite[Page 71]{MR1076249}).

\begin{lemma} \label{lem:mainconditions}
Let $(A,B) \in T_n(\Q_p) \times T_n(\Z_p)$. Then
$(A,B) \in \setU_p$ if and only if each of the following holds.
\begin{enumerate}
\item \label{item:maincon1} $\det(A) \neq 0$ and the rows $a_1,\dotsc, a_n$ of the matrix $A$ satisfy
\begin{eqnarray*}
\text{ for } 1 \leq i < j \leq n \text{ there exist } Y_{ij}^1, \ldots, Y_{ij}^{i-1} \in \Z_p \text{ such that }  \\
\kappa (a_i, a_j) = \mu^{(i-1)} (\lambda(a_1, Y_{ij}^1) ,
  \ldots, \lambda(a_{i-1}, Y_{ij}^{i-1} )).
\end{eqnarray*}
\item \label{item:maincon3} 
$\det(B) \neq 0$ and the rows $b_1, \ldots, b_n$ of the matrix $B$ satisfy
\begin{eqnarray*}
\text{ for } 1 \leq i < j \leq n \text{ there exist } Y_{ij}^1, \ldots, Y_{ij}^{i-1} \in \Z_p \text{ such that }  \\
\kappa (b_i, b_j) = \mu^{(i-1)} (\lambda(b_1, Y_{ij}^1) ,
  \ldots, \lambda(b_{i-1}, Y_{ij}^{i-1} )).
\end{eqnarray*}
Moreover, let ${\bf y} : \Z_p^n \to G(\Q_p)$ be the function defined by
$$
{\bf y} ({\bf d}) := {\bf x}(a_1)^{d_1} {\bf x}(a_2)^{d_2} \cdots {\bf x}(a_n)^{d_n}.
$$
Then there exists vectors $c_1, \ldots, c_i \in \Z_p^n$ such that for all $i =1 , \ldots, n$, 
$$
{\bf y}(c_i) = {\bf x}(b_i).
$$
\item \label{item:maincon4} $|\det(B)|_p$ is maximal over all $B$ satisfying \ref{item:maincon3}.
%\item \label{item:maincon4} 
%Let ${\bf y}$ be as in Item \ref{item:maincon3} and ${\bf z}: \Z_p^n \to G(\Q_p)$ be the function defined by
%$$
%{\bf z}({\bf d}) = {\bf x}(b_1)^{d_1} {\bf x}(b_2)^{d_2} \cdots {\bf x}(b_n)^{d_n}.
%$$
%Then for all $s_1, \ldots, s_n, t_1, \ldots, t_n \in \Z_p^n$,
%there exists $r_1, \ldots, r_n \in \Z_p^n$ such that
%$$
%{\bf y} (s_1, \ldots, s_n) = {\bf x}(t_1, \ldots, t_n)
%\implies {\bf z}(r_1, \ldots, r_n) = {\bf x}(t_1, \ldots, t_n).
%$$
%Equivalently, $|\det(B)|_p$ is maximal over all $B$ satisfying \ref{item:maincon3}.

\end{enumerate}
\end{lemma}

\begin{proof} 
  Let $(A, B) \in \setU(\Delta)$. Then Conditions
  \ref{item:maincon1} and \ref{item:maincon3} are satisfied by
  Lemma~\ref{lem:goodbasis}. To see
  Condition~\ref{item:maincon4}, note that $|\det(B)|_p = [G(\Z) :
  G(\Z) \cap \Delta]$, while any matrix $B'$ satisfying
  Condition~\ref{item:maincon3} determines a subgroup of $G(\Z_p) \cap \overline{\Delta}$ and hence satisfies $|\det(B')|_p \leq
  [G(\Z) : G(\Z) \cap \Delta]$. 

  Let $(A, B) \in \setU_p$. Let $H$ be the closed subgroup of
  $G(\Q_p)$ generated by ${\bf x}(a_i)$ where $a_i$ are the rows or
  $A$ and let $\Delta$ be the subgroup in $\CL_{p^\ell}(G)$
  corresponding to $H$ given by Lemma \ref{lem:correspondence}. The conditions above ensure
  $(A, B) \in \setU(\Delta)$.
\end{proof}

Our parametrization of arithmetic lattices gives a nice formula for the commensurability index:

\begin{lemma} \label{lem:comindex}
Let $\Delta \in \CL(G)$. Then for any $(A, B) \in \setU(\Delta)$, we have
$$
\comIndex(\overline{\Delta}, G(\Z_p)) = |\det (B)^2\det(A)^{-1}|_p^{-1}.
$$
\end{lemma}

\begin{proof}
A direct application of the results of \cite[pp. 198]{MR943928} shows
\begin{equation} \label{eqn:comindex1}
  [G(\Z_p):G(\Z_p) \cap \overline{\Delta}] = |\det(B)|_p^{-1}.
\end{equation}

Let $C$ be the matrix whose $i$th row $c_i$ satisfies
${\bf y}(c_i) = {\bf x}(b_i)$, where ${\bf y}$ is the 
function defined in Lemma \ref{lem:mainconditions}. 
Note that $C \in T_n(\Z_p)$. 
Because $\overline{\Delta} = {\bf y}(\Z_p)$, it follows as above from
\cite[pp. 198]{MR943928} that
\begin{equation}\label{eqn:comindex2}
  [\overline{\Delta} : G(\Z_p) \cap \overline{\Delta}] = |\det(C)|_p^{-1}.
\end{equation}

Using the fact that the rows of $B$ are coordinate vectors of a good
basis, computation shows that
\[
  b_i  = ( q_{i1}, \dotsc, q_{i(i-1)}, a_{ii}c_{ii}, 0,\dotsc, 0 ),
\]
where $q_{ij} \in \Q_p$ for all $j< i$ and $a_{ii}$ and $c_{ii}$ are
the respective diagonal elements of $A$ and $C$. It follows that
\begin{equation} 
  |\det(B)|_p^{-1} = |\det(AC)|_p^{-1}.
\end{equation}
The desired result follows.
\end{proof}

For $n \in \N$, let $\nu$ be the Haar measure on $\Q_p^n$ normalized such that $\nu(\Z_p^n) = 1$.

\begin{lemma} \label{lem:volcomp}
For a matrix $M \in T_n(\Q_p)$, let $f(M) := \prod_{i=1}^n m_{11} \cdots m_{ii}$.
Let $\Delta \in \CL(G)$. Then $\setU(\Delta)$ is an open subset of $T_n(\Q_p) \times T_n(\Z_p)$ and 
$$\nu(\setU(\Delta)) = (1-p^{-1})^{2n}  | f(A) f(B) |_p, $$
for any $(A,B) \in \setU(\Delta)$.
\end{lemma}

\begin{proof}
  For $\Delta \in \CL(G)$ and a ring $R$, let $\mathcal{M}_R(\Delta)$ be the set of
  matrices in $T_n(R)$ such that the rows $m_1, \ldots, m_n$ of $M$
  satisfy ${\bf x}(m_1), \ldots, {\bf x}(m_n)$ are a good basis for
  $\overline{\Delta}$.  Notice that
  $\setU( \Delta ) = \mathcal{M}_{\Q_p}( \Delta) \times
  \mathcal{M}_{\Z_p}(G(\Z) \cap \Delta)$.  Hence, that $\setU(\Delta)$
  is an open subset of $T_n(\Q_p) \times T_n(\Z_p)$ is a
  straightforward application of the first part of the proof of
  \cite[Lemma 2.5, pp. 198]{MR943928}.

Moreover, the second part of the proof of \cite[Lemma 2.5, pp. 198]{MR943928} and Fubini's Theorem gives:
\begin{eqnarray*}
\int_{\setU(\Delta)} 1 d\nu &=& \int_{\mathcal{M}_{\Q_p}({\Delta}) \times \mathcal{M}_{\Z_p}(G(\Z) \cap \Delta)} 1 d\nu \\
&=& \int_{\mathcal{M}_{\Q_p}({\Delta})}\int_{\mathcal{M}_{\Z_p}(G(\Z) \cap {\Delta})} 1 d\nu \\
&=&  \int_{\mathcal{M}_{\Q_p}{\Delta})} (1-p^{-1})^n |f(B)|_p d\nu \\
&=&  (1-p^{-1})^{2n} |f(A)f(B)|_p .
\end{eqnarray*}

\end{proof}

\begin{proposition} \label{prop:zetaintegral}
Let $f$ be as in Lemma \ref{lem:volcomp}.
We have
$$
\zeta_{G,p}(s) = \frac{1}{(1-p^{-1})^{2n}} \int_{(A, B) \in \setU_p} 
|f(A)f(B)|_p^{-1} |\det(B)^2\det(A)^{-1}|_p^{s}  d\nu.
$$
\end{proposition}

\begin{proof}
By Lemma \ref{lem:volcomp}, for any $(A,B) \in \setU(\Delta)$, the integrand evaluates to:
\begin{eqnarray*}
\nu (\setU(\Delta))^{-1} (\comIndex(\Delta, G(\Z))^{-s}.
\end{eqnarray*}
By decomposing $\setU_p$ into disjoint open sets,
$$
\setU_p = \bigcup_{k=1}^\infty \left(  \bigcup_{\Delta \in \CL_{p^k}(G)} \setU(\Delta) \right),
$$
using Lemma \ref{lem:volcomp} we compute:
\begin{eqnarray*}
&& \frac{1}{(1-p^{-1})^{2n}} \int_{\setU_p} \left( |f(A)f(B)|_p^{-1}
   |\det(B)^2 \det(A)^{-1}|_p^{s}  \right) d\nu \\
&=& \sum\limits_{k=1}^\infty  \left( \sum_{\Delta \in \CL_{p^k}(G)} \left( \int\limits_{\setU(\Delta)} \nu (\setU(\Delta))^{-1} (\comIndex(\Delta, G(\Z))^{-s} d \nu \right)\right) \\
&=& \sum\limits_{k=1}^\infty  \left( \sum_{\Delta \in \CL_{p^k}(G)} \left( \int\limits_{\setU(\Delta)} \nu (\setU(\Delta))^{-1} p^{-ks
} d \nu \right)\right) \\
&=& \sum\limits_{k=1}^\infty  \left( \sum_{\Delta \in \CL_{p^k}(G)}p^{-ks}\right) = \sum_{k=1}^\infty \frac{\comGrowth_{p^k}(\CL(G)) }{p^{ks}} = \zeta_{G,p}(s).
\end{eqnarray*}
\end{proof}

We are now ready to prove our main result:
\begin{proof}[Proof of Theorem \ref{thm:main}]
We first show that the local commensurability zeta functions converge for sufficiently large $s$. 
Given $\Delta \in \CL_{p^k} (G)$, for any $\gamma \in \Delta$ we have $\gamma^{p^k} \in G(\Z_p)$ so  $\Delta \leq \Gamma_k := (G(\Z))^{p^{-k}}$.
Then by Item \ref{item:rootsc} of Lemma \ref{lem:roots} there exists $D$, depending only on $G$, such that
$$
|\Gamma_k : G(\Z)| \leq p^{Dk}.
$$
Hence, any $\Delta \in \CL_{p^k}(G)$ is a subgroup of $\Gamma_k$ of index at most $p^{(1+D)k}$, giving
\begin{equation}
\comGrowth_{p^k}(\CL(G)) \leq s_{p^{(1+D)k}} (\Gamma_k) \leq s_{p^{(1+D)k}} (N). \label{ineq:last}
\end{equation}
where $s_m(G)$ is the subgroup growth function (the number of subgroups of index at most $m$) and $N$ is the free nilpotent group of class and rank equal to that of $G(\Z)$.
It follows from \cite{MR943928} that $s_{p^{(1+D)k}}(N) \leq p^{\alpha k}$ for some fixed $\alpha$ that does not depend on $p$. Hence, Inequality (\ref{ineq:last}) gives that $\zeta_{G,p}(s)$ is finite for $s > M$ where $M$ does not depend on $p$.

The integrand in Proposition \ref{prop:zetaintegral} and $\setU_p$ are both $L_p$-definable in the sense of
\cite{MR1076249} (see also \cite{MR751129}) by Lemma \ref{lem:mainconditions}.
Thus, by \cite[Theorem 22]{MR1076249} we have that $\zeta_{G,p}(s)$ is a rational function with numerator and denominator degrees bounded by a constant depending only on $G$.
\end{proof}

\bibliography{refs}
\bibliographystyle{amsalpha}

%---------------------------------------------------------------------------

\noindent
Khalid Bou-Rabee \\
Department of Mathematics, City College of New York \\
E-mail: kbourabee@ccny.cuny.edu \\

\noindent
Daniel Studenmund \\
Department of Mathematics, University of Notre Dame \\
E-mail: dstudenm@nd.edu \\

\end{document}